\documentclass[12pt]{article} 
\usepackage{amsmath}
\usepackage{amssymb}
\usepackage{amsthm}
\usepackage{mathtools}

\usepackage[utf8]{inputenc}
\usepackage[english]{babel}
\usepackage{color}
\usepackage{hyperref}
\usepackage{url}

\usepackage{wrapfig}

\newtheorem{theorem}{Theorem}[section]
\newtheorem{conjecture}[theorem]{Conjecture}
\newtheorem{corollary}[theorem]{Corollary}
\newtheorem{example}[theorem]{Example}

\newtheorem{proposition}[theorem]{Proposition}

\newtheorem{classictheorem}{Theorem}

\title{Angle chains and pinned variants}

\date{\today}

\author{
Eyvindur Ari Palsson\thanks{Department of Mathematics, Virginia Tech, Blacksburg, VA 24061. {\sl palsson@vt.edu}. Supported by Simons Foundation Grant \#360560.}
\and
Steven Senger\thanks{Department of Mathematics, Missouri State University, Springfield, MO 65897. {\sl stevensenger@missouristate.edu}.}
\and
Charles Wolf\thanks{Department of Mathematics, University of Rochester, Rochester, NY 14627.
{\sl charles.wolf@rochester.edu}.}}

\begin{document}
\maketitle
\begin{abstract}
We study a variant of the Erd\H os unit distance problem, concerning angles between successive triples of points chosen from a large finite point set. Specifically, given a large finite set of $n$ points $E$, and a sequence of angles $(\alpha_1,\ldots,\alpha_k)$, we give upper and lower bounds on the maximum possible number of tuples of distinct points $(x_1,\dots, x_{k+2})\in E^{k+2}$ satisfying $\angle (x_j,x_{j+1},x_{j+2})=\alpha_j$ for every $1\le j \le k$ as well as pinned analogues.
\end{abstract}

\section{Introduction}
\subsection{Background}
In \cite{Erd46}, Erd\H os introduced two popular problems in discrete geometry, the unit distance problem and the distinct distances problem. Given a finite point set in the plane, the unit distance problem asks how often a single distance can occur between pairs of points, while the distinct distances problem asks how many distinct distances must be determined by pairs of points. See \cite{BMP, GIS} for surveys of these and related problems. The distinct distances problem was resolved in 2010 by Guth and Katz in \cite{GK} in the plane, but remains open in higher dimensions. By contrast, the unit distance problem has not seen any progress since the work of Spencer, Szemer\' edi, and Trotter \cite{SST84} in $\mathbb{R}^2$, while recent progress has been made in $\mathbb{R}^3$ by Zahl \cite{Zahl19}. In $\mathbb{R}^d$ with $d\geq 4$ the unit distance problem becomes trivial without further restrictions due to the celebrated Lenz example, which we describe in detail below (Theorem \ref{LenzExample}). A variant of these problems involves fixing one of the points from which the distances are counted. These are referred to as pinned variants. The same bounds are conjectured for the pinned Erd\H{o}s distinct distances problem as the unpinned version with the best partial results obtained by Katz and Tardos in \cite{KT04} in the plane. The pinned version of the unit distance problem is trivial since all but one of the points can be placed on a circle around the remaining point.

Another variant of the family of problems proposed by Erd\H os that is important to this paper involves point configurations determined by distances measured between more than two points. One of the most commonly studied configurations is a $(k+1)$-tuple of points where each distance between the $k$ consecutive pairs of points is specified. In the context of the distinct distances problem this was solved in the style of the Guth-Katz argument by Misha Rudnev in \cite{Rudnev19} and Jonathan Passant \cite{Passant20}. Analogous problems have been considered in both finite fields \cite{BCCHIP16} as well as the continuous setting for both pinned and unpinned Falconer type problems for chains \cite{BIT, OT20}. The study of such problems in the unit distance setting was initiated by the first two authors and Sheffer in \cite{PSS} and their results were improved upon by Frankl and Kupavskii \cite{FK19}. These chain variants of the original problems include the original problems as special cases, so it was surprising in the unit distance setting that sharp results were obtained for two-thirds of all chains, while the remaining cases only miss by as much as the best results currently available for the unit distance problem itself. This unexpected development has brought attention to these types of problems, with some follow up work done both on more complicated relationships, as well as replacing distances by dot products \cite{GPRS20, KMS20}.
\subsection{Angles}
The final variant of the family of problems proposed by Erd\H os important to this paper is replacing the distance between two points by the angle made by three points. The important paradigm shift here is that the base configuration now depends on three points as opposed to only two points as in the case of other widely-studied objects, such as distances, dot products, and directions. This poses a unique challenge in that the incidence results ubiquitous in this area tend to look at level sets with respect to a single point, such as circles which encode a fixed distance to a given point. However, in the case of angles, the level set of a single point is already rather complicated, and looking at how level sets of how multiple pairs of points interact is even more so.

To keep track of the various parameters, we borrow notation from related problems on distances in \cite{BIT, PSS}. Specifically, if we fix a $k$-tuple of real numbers, $(\alpha_1, \alpha_2, \dots, \alpha_k)\in(0,\pi)^k$, then a {\bf $k$-chain} of that type is a $(k+2)$-tuple of points, $(x_1, x_2, \dots, x_{k+2}),$ such that for all $j=1,\dots, k,$ we have $\angle(x_j,x_{j+1},x_{j+2})=\alpha_j.$ For example, if we fix a triple of real numbers, $(\alpha, \beta, \gamma)$, a 3-chain of that type will be a set of five points, where the angle determined by the first three points is $\alpha$, the angle determined by the middle three points is $\beta$, and the angle determined by the last three points is $\gamma.$ Given a large finite point set $E$ and a $k$-tuple of angles, $(\alpha_1, \alpha_2, \ldots, \alpha_k),$ we denote the set of $k$-chains determined by $(k+2)$-tuples of points in $E$ by
$$\Lambda_k(E; \alpha_1, \alpha_2, \ldots, \alpha_k):=\left\{ (x_1, x_2, \dots, x_{k+2})\in E^{k+2}:\angle(x_j,x_{j+1},x_{j+2})=\alpha_j\right\}.$$
When context is clear, we suppress the angles and just write $\Lambda_k(E).$ In some cases, we will have all of the $\alpha_j$ equal to a fixed $\alpha$, and will refer to a chain as an $\alpha$ angle $k$-chain. Also, we assume that angles are not integer multiples of $\pi$, as then we could just arrange points along a line and get $n^{k+2}$ instances of a $k$-chain whose angles are multiple of $\pi.$ Moreover, we will also assume that $k$ is like a constant compared to the number of points in a given set. If two quantities, $X(n)$ and $Y(n)$, vary with respect to some natural number parameter, $n$, then we write $X(n) \lesssim Y(n)$ if there exist constants, $C$ and $N$, both independent of $n$, such that for all $n> N$, we have $X(n)\leq CY(n)$. If $X(n) \lesssim Y(n)$ and $Y(n) \lesssim X(n)$, we write $X(n) \approx Y(n).$

In \cite{PachSharir}, Pach and Sharir gave the following upper bound on the size of $\Lambda_1(E),$ the number of triples of points determining a fixed angle. They also showed that their result is sharp for some angles, so we cannot expect to do better in general.
\begin{classictheorem}\label{PS}
Given a large finite point set $E$ of $n$ points in the plane, $$|\Lambda_1(E)|\lesssim n^2 \log n.$$
\end{classictheorem}

This work was continued in higher dimensions by Apfelbaum and Sharir \cite{ApfelbaumSharir}. They proved the following two results in three and four dimensions.

\begin{classictheorem}\label{AS3}
Given a large finite point set $E$ of $n$ points in $\mathbb R^3$, $$|\Lambda_1(E)|\lesssim n^\frac{7}{3}.$$
\end{classictheorem}

This estimate is sharp in the case that the angle in question is $\frac{\pi}{2}.$ To convey the four-dimensional bound, we use the function $\beta(n)$ which grows extremely slowly, as it is defined using the inverse Ackermann function. To be completely rigorous, we can write $\beta(n)\lesssim n^\epsilon$ for any $\epsilon >0$, but in practice, it is essentially a constant.

\begin{classictheorem}\label{AS4}
Given a large finite point set $E$ of $n$ points in $\mathbb R^4$, for $\alpha \neq \frac{m\pi}{2}$ for any integer $m$, $$|\Lambda_1(E;\alpha)|\lesssim n^\frac{5}{2}\beta(n).$$
\end{classictheorem}

In the case that the angle under consideration is $\frac{\pi}{2}$, there is a construction that yields $\approx n^3$ triples that determine a right angle. This is not surprising, as the angle $\frac{\pi}{2}$ is related to points with dot product zero, which exhibit some distinct behavior in their own right. Therefore, without special assumptions on either the point set or the angles, the question becomes trivial in higher dimensions. We discuss this in greater detail in Section \ref{4d+chains}.

\section{Main results}

In this note, we extend the aforementioned results to $k$-chains of angles, pinned angles and pinned $k$-chains of angles.

\subsection{Angle chains in $\mathbb{R}^2$}

We first obtain the following upper bounds on angle chains in the plane.

\begin{theorem}\label{twoChains}
Given a large finite point set $E$ of $n$ points in the plane, and a $k$-tuple of angles $(\alpha_1,...,\alpha_k),$
$$|\Lambda_k(E)|\lesssim \left\{\begin{tabular}{c r}$n^{\frac{k-1}{2}+2}\log n,$ & $n$ odd\\ $n^{\frac{k}{2}+2},$ & $n$ even \end{tabular}\right.$$
\end{theorem}

This result is sharp up to logarithms, as the following family of lower bounds will show.

\begin{theorem}\label{railroad}
Given $(\alpha_1, \dots, \alpha_k),$ there exists a set of $n$ points in the plane that has $\gtrsim n^{\left\lfloor\frac{k}{2}\right\rfloor+2}$ angle $k$-chains of type $(\alpha_1, \dots, \alpha_k).$
\end{theorem}

These theorems are proved in Section \ref{2dchains}.

\subsection{Angle chains in $\mathbb{R}^3$}

In contrast to the planar case, a wide open problem reveals itself for 2-chains of right angles. We currently have no nontrivial bounds for chains of non-right angles in three dimensions. In general we are only able to obtain the trivial upper bound of $\lesssim n^{\frac{10}{3}}$ obtained by observing there are $\lesssim n^{\frac{7}{3}}$ choices for the first three points using the result of Apfelbaum and Sharir in Theorem \ref{AS3}, and then $n$ choices for the fourth and final point. In Subsection \ref{2chainsin3d} we obtain many improvements on this result under further conditions on our point set. Collectively, these partial results point towards the following conjecture.

\begin{conjecture}\label{2chainsConj}
Given a large finite point set $E \subset \mathbb R^3$ of $n$ points we have that for right angles,
$$|\Lambda_2(E)| \lesssim n^{3}.$$
\end{conjecture}

Note that this conjecture matches the lower bound that we get by embedding the appropriate 2-chain construction from Theorem \ref{railroad} in $\mathbb R^3.$ In Subsection \ref{3+chainsin3d} we obtain non-trivial results for right angle $k$-chains with $k\geq 3$ and develop an induction mechanism that generates bounds for right angle chains of arbitrary length. If the conjecture above were to be confirmed, or even if progress is made toward it in the general case, it would immediately lead to improved bounds on some cases of longer right angle chains. We now summarize our results for point sets in $\mathbb R^3,$ which build on Theorem \ref{AS3}.

\begin{theorem}\label{threeChains}
Given a large finite point set $E \subset \mathbb R^3$ of $n$ points, we have that for right angles,
$$|\Lambda_1(E)| \lesssim n^\frac{7}{3}, |\Lambda_2(E)| \lesssim n^\frac{10}{3}, |\Lambda_3(E)| \lesssim n^4, |\Lambda_4(E)| \lesssim n^\frac{13}{3}, |\Lambda_5(E)| \lesssim n^5,$$
$$\text{ and }|\Lambda_k(E)| \lesssim n^{\frac{1}{3}\left(19+\left\lfloor\frac{7(k-7)}{4}\right\rfloor\right)},\text{ for }k\geq 6.$$
\end{theorem}

Proofs of these results are found in Section \ref{3dchains}.

\subsection{Higher dimensions and pinned variants}
In Section \ref{4d+chains} we recall the Lenz example in full detail, as it is the root from which the other higher dimensional results sprang. We then prove that the angle chain question becomes trivial without significant further restrictions in $\mathbb{R}^d$ with $d\geq 6$. We also give some partial results in five dimensions, where the problem does not appear to be trivial. Of course, if a given construction exists in some dimension $d,$ then it can be embedded into a higher dimensional space.

Section \ref{pinnedSection} contains a number of estimates when we fix or ``pin" one of the points in question. For pinned variants we first observe that different behavior is possible depending on which point is pinned, unlike in the single distance case where the roles are symmetric. While we focus on pinning the first point for most of our estimates, the following construction makes this distinction quite explicit when compared to the other pinned results we consider. This shows that the trivial bound of $n^2$ can be achieved by a straightforward arrangement in the case of angles with a pinned middle point.

\begin{proposition}\label{middlePinned}
In $\mathbb R^d,$ with $d\geq 2$, for any large, finite $n,$ there exists a set of $n$ points with $\approx n^2$ triples of points determining any angle $\alpha$, that share the same middle point.
\end{proposition}

In contrast to the previous result, we get much different results by pinning the first point.

\begin{theorem}\label{th:2dpin}
For any $n$ points in $\mathbb{R}^2$, and any fixed angle $0<\alpha<\pi$, there are at most $\lesssim n^{4/3}$ triples of points with angle $\alpha$ starting from the origin, and this is sharp.	
\end{theorem}

For angle chains in $\mathbb{R}^2$ pinned at the endpoint we establish sharp bounds up to logarithmic terms, similar to the unpinned setting.

\begin{theorem}\label{th:2dchainpinupper}
	
	For any $n$ point set in $\mathbb{R}^2$, integer $k\ge 2$ and angles $(\alpha_1,...,\alpha_k)$, the number of $k$-chains of type $(\alpha_1,...,\alpha_k)$ starting from the origin is  	
	$$\lesssim \begin{cases}
	n^{\frac{k-1}{2}+1}\log n,& k\text{ is odd}\\
	n^{\frac{k}{2}+1},& k\text{ is even}
	\end{cases}$$
\end{theorem}

\begin{theorem}\label{th:2dchainpinlower}
	For any ordered set of angles $(\alpha_1,...,\alpha_k)$, there is a set of $\approx n$ points in $\mathbb{R}^2$ forming $\gtrsim n^{\lfloor{\frac k2}\rfloor+1}$ instances of $k$-angle chains of type $(\alpha_1,...,\alpha_k)$ starting at the origin.  
\end{theorem}

In $\mathbb{R}^3$ we show that for a single right angle the problem is already trivial without further restrictions. We also show that the pinned problem is trivial for longer chains in $\mathbb R^6.$ The precise statements and proofs are in Section \ref{pinnedSection}.

\section{Angle chains in $\mathbb{R}^2$}\label{2dchains}

We first recall the statement of the Theorem \ref{twoChains}.
\vspace{2mm}

\noindent {\bf Theorem \ref{twoChains}.}
\emph{Given a large finite point set $E$ of $n$ points in the plane,
$$|\Lambda_k(E)|\lesssim \left\{\begin{tabular}{l l}$n^{\frac{k-1}{2}+2}\log n,$ & $k$ odd\\ $n^{\frac{k}{2}+2},$ & $k$ even \end{tabular}\right.$$}

\begin{proof}
Given a large finite point set $E$, and a type of $k$-chain, $(\alpha_1, \alpha_2, \dots, \alpha_k),$ we seek to bound the number of $(k+2)$-tuples of points from $E$, $(x_1,\ldots,x_{k+2})$ with the property that $\angle(x_i, x_{i+1}, x_{i+2})=\alpha_i.$ We handle the cases of $k$ even and $k$ odd separately.

{\bf Even $k$:} Pick $x_1$ and $x_2.$ We have $\approx n^2$ such choices. Now pick $x_4$, and notice that the there is a unique location for $x_3$ so that $\angle(x_1, x_2, x_3) = \alpha_1$ and $\angle(x_2, x_3, x_4)= \alpha_2.$ Continuing inductively, there are $n$ choices for each subsequent even indexed point, $x_{2j}.$ Each such choice, will fix the location of $x_{2j-1}$, as we have already chosen $x_{2j-3}$ and $x_{2j-2}.$ This yields an upper bound of $n^{\frac{k}{2}+2}.$

{\bf Odd $k$:} We use Theorem \ref{PS} to get a bound of $n^2 \log n$ choices for the first triple of points, $(x_1, x_2, x_3)$ so that $\angle(x_1, x_2, x_3)=\alpha_1.$ Then we proceed in a manner similar to the even case. That is, continuing inductively, there are $n$ choices for each subsequent odd indexed point, $x_{2j+1}.$ Each such choice, will fix the location of $x_{2j}$, as we have already chosen $x_{2j-2}$ and $x_{2j-1}.$ This yields an upper bound of $n^{\frac{k-1}{2}+2}\log n.$
\end{proof}

We next recall the statement of the Theorem \ref{railroad}.
\vspace{2mm}

\noindent {\bf Theorem \ref{railroad}.}
\emph{Given $(\alpha_1, \dots, \alpha_k),$ there exists a set of $n$ points in the plane that has $\gtrsim n^{\left\lfloor\frac{k}{2}\right\rfloor+2}$ angle $k$-chains of type $(\alpha_1, \dots, \alpha_k).$}

\begin{proof}
Given $(\alpha_1, \dots, \alpha_k),$ set $m=\lfloor 2n/k \rfloor.$ Arrange $m$ points, $p_1, \dots, p_m,$ in order away from the origin, along the $x$-axis, which we will call $\ell_1$, then draw the line $\ell_2$ so that it intersects the line $\ell_1$ at an angle of $\alpha_1-\alpha_2$ at the origin. Now for each of the points $p_j$ on $\ell_1,$ except for $p_m,$ the one furthest from the intersection of $\ell_1$ and $\ell_2$, put a point, $p_{m+j}$, on $\ell_2$ so that $\angle(p_{j'},p_j, p_{m+j})=\alpha_1$ for any $j'>j.$ Notice that by construction, we will also have that $\angle(p_j, p_{m+j}, p_{m+j''})=\alpha_2$ for any $j''<j.$ Now rotate the whole set so that $\ell_2$ coincides with the $x$-axis, and draw $\ell_3$ so that it intersects $\ell_2$ at the origin in an angle of $\alpha_3-\alpha_4$, and continue. At each stage, choose points so that they don't overlap with other points just in case the set wraps around the origin after a number of rotations. Notice that we can form the desired number of angle $k$-chains by picking $x_1$ and $x_2$ from $\approx n^2$ pairs of points from $\ell_1$, then a fixed point $x_3$, determined by $x_2$, from $\ell_2$, then choosing $x_4$ from $\approx n$ points on $\ell_2$ so that $\angle(x_2,x_3,x_4)=\alpha_2,$ and so on.

\begin{centering}
\includegraphics[scale=.7]{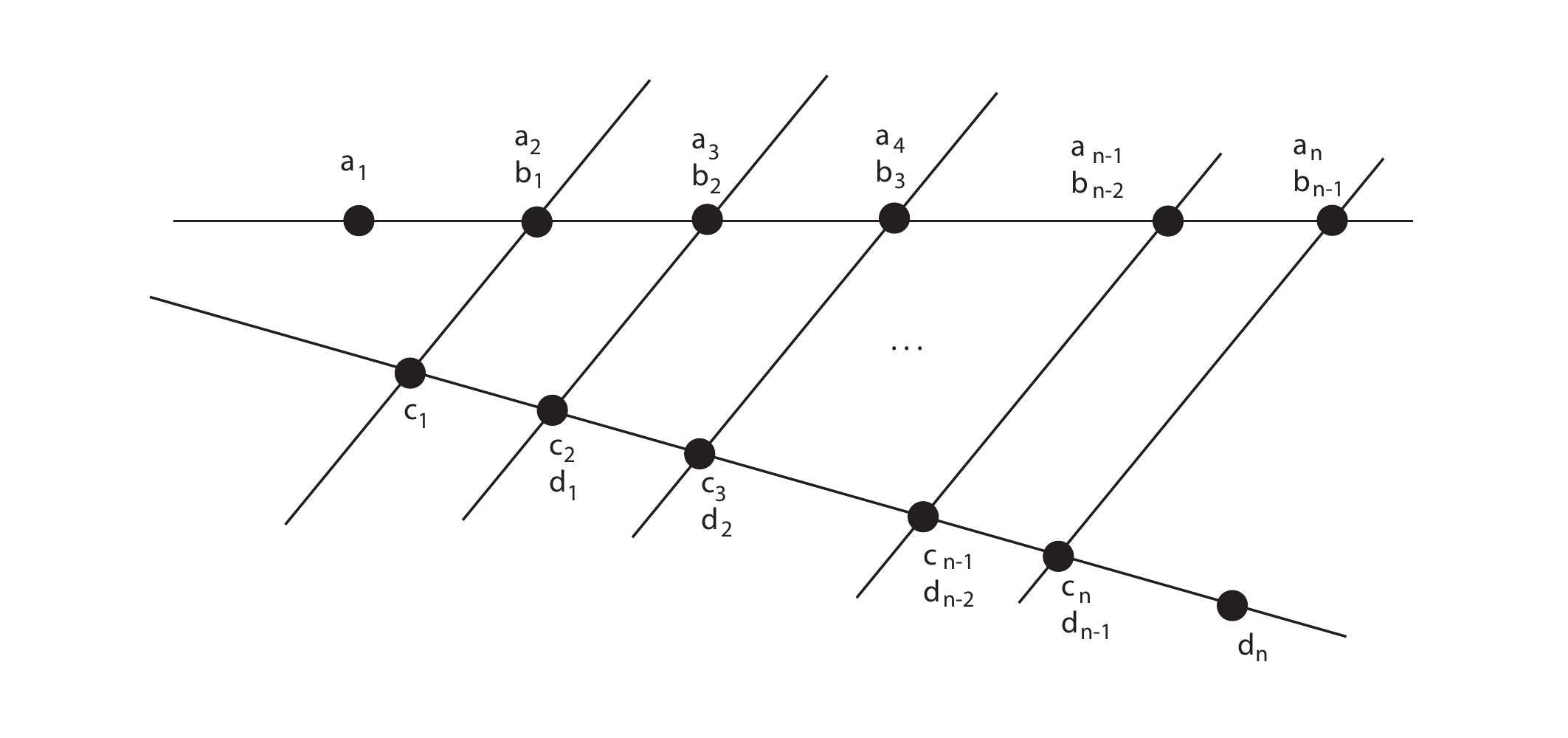}
{\bf Figure 1:} This shows a set of points arranged on two lines that determine $\gtrsim n^3$ instances of 2-chains of the form $(a_i,b_j,c_j,d_k)$ of some non-right angles.
\end{centering}
\end{proof}

\section{Angle chains in $\mathbb{R}^3$}\label{3dchains}
\subsection{Point-line incidences}
As noted in the Introduction, the maximum number of angles determined by a large, finite set of $n$ points in $\mathbb{R}^3$ is $\approx n^{7/3}$. Despite considerable effort, we could not improve the upper bound on $|\Lambda_2(E)|$ for general sets $E\subset \mathbb R^3$ of $n$ points. While there are many results bounding points and various algebraic varieties (Adam Sheffer has an extensive exposition in \cite{Sheffer}), we were unable to control incidences of points and planes sufficiently to beat the trivial estimate in general.

Here we record some partial results toward Conjecture \ref{2chainsConj}. We first state the celebrated Szemer\'edi-Trotter point-line incidence estimate from \cite{ST83}.
\begin{classictheorem}\label{szemTrot}
Given a large finite set of $n$ points and $m$ lines in $\mathbb R^2,$ the number of point-line incidences is bounded above by
$$\lesssim n^\frac{2}{3}m^\frac{2}{3}+n+m.$$
\end{classictheorem}
One consequence of this estimate is the following.
\begin{classictheorem}\label{richLines}
Given a large, finite set of $n$ points in $\mathbb R^d$, with $d\geq 2,$ and a number $r \geq 2,$ the number of lines with at least $r$ points from the set on them is
$$\lesssim \frac{n}{r}+\frac{n^2}{r^3}.$$
\end{classictheorem}
This holds in higher dimensions because with a finite set of points, one can always safely project points to some plane, and apply Theorem \ref{szemTrot} there.

\subsection{Counting 2-chains in $\mathbb{R}^3$}\label{2chainsin3d}

\begin{theorem}
	Consider a set $E$ of $n$ points in $\mathbb{R}^3$.  If every plane contains at most $p$ points of $E$, then the number of right angle 2-chains is $\lesssim pn^{7/3}$.
\end{theorem}
\begin{proof}
	Consider a right angle 2-chain of points $(x_1,x_2,x_3,x_4)$. By Theorem~\ref{AS3} we can choose $(x_1,x_2,x_3)$ in $n^{7/3}$ ways. Since $\angle (x_2, x_3, x_4)$ is a right angle, $x_4$ can lie in the plane containing $x_3$ with normal vector $\overrightarrow{x_2x_3}$.  Since each plane has at most $p$ points of $E$, the total number of right angle 2-chains is $\lesssim pn^{7/3}$. 
\end{proof}

In \cite{ApfelbaumSharir} they show the construction obtaining the asymptotically tight bound of $\approx n^{7/3}$ is the lattice cube $[1,...,n^{1/3}]^3$.  Since the number of points in any plane is at most $n^{2/3}$, we get the following corollary: 

\begin{corollary}
	The lattice cube $[1,...,n^{1/3}]^3$ has at most $\lesssim n^3$ instances of right angle 2-chains.
\end{corollary}

\begin{theorem}\label{th:parallel}
	The number of right angle 2-chains where the $\overline{x_1x_2}$ line is parallel to the $\overline{x_3x_4}$ line is $\lesssim n^3.$
\end{theorem}
\begin{proof}
	We can choose $x_1, x_2$ in $\approx n^2$ ways.  This also fixes the direction for $\overline{x_3x_4}$.  We can choose $x_4$ in $\approx n$ ways, and then there is only one choice for $x_3.$
\end{proof}

\begin{theorem}
	For a set of $n$ points in $\mathbb{R}^3$, suppose a right angle 2-chain is formed where the $\overline{x_1x_2}$ line has at least $j$ points, the $\overline{x_3x_4}$ line has at least $k$ points, and these lines are not parallel.  The number of ways to choose these 2-chains is $$\lesssim\left(\dfrac{n^2}{j^2}+n\right)\left(\dfrac{n^2}{k^2}+n\right).$$
\end{theorem}

In particular, if $j$ and $k$ are both greater than $n^\frac{1}{4}$ or one of them is greater than $n^\frac{1}{2}$, then this term is $\lesssim n^3$.

\begin{proof}
By Theorem \ref{richLines}, the number of ways to choose a line with at least $i$ points is $\lesssim\left(\frac{n^2}{i^3}+\frac ni\right).$ We can choose the two lines in $\lesssim\left(\frac{n^2}{j^3}+\frac nj\right)\left(\frac{n^2}{k^3}+\frac nk\right)$ ways, plus choosing $x_1$ and $x_2$ in $j\cdot k$ ways.  Since the $\overline{x_1x_2}$ line is not parallel to the $\overline{x_3x_4}$ line, there is at most one choice of $x_2$ and $x_3$ to form a right angle 2-chain.  So the number of right angle 2-chains is $\lesssim\left(\frac{n^2}{j^2}+n\right)\left(\frac{n^2}{k^2}+n\right).$
\end{proof}

\begin{theorem}
	For a set of $n$ points in $\mathbb{R}^3$, suppose a right angle 2-chain is formed where the $\overline{x_2x_3}$ line has at least $m$ points, and the lines $\overline{x_1x_2}$ and $\overline{x_3x_4}$  are not parallel.  The number of ways to choose these 2-chains is $$\lesssim \left(\dfrac{n^4}{m^3}+\dfrac {n^3}m\right).$$
\end{theorem}

In particular, if $m\ge n^{1/3}$, then the number of 2-chains is $\lesssim n^3.$

\begin{proof}
	The number of ways to choose a $\overline{x_2x_3}$ line with at least $m$ points is $\lesssim \left(\dfrac{n^2}{m^3}+\dfrac nm\right).$ Having chosen this line, we can choose $x_1$ and $x_4$ in $\approx n^2$ ways. Since the lines $\overline{x_1x_2}$ and $\overline{x_3x_4}$  are not parallel, there are is at most one choice of $x_2$ and $x_3$.  So the number of these 2-chains is $\lesssim \left(\dfrac{n^4}{m^3}+\dfrac {n^3}m\right).$ 
\end{proof}

\subsection{Longer chains in $\mathbb{R}^3$}\label{3+chainsin3d}
\begin{theorem}\label{3chains3}
Given a large finite point set $E$ of $n$ points in $\mathbb R^3$,
$$|\Lambda_3(E)|\lesssim n^4.$$
\end{theorem}

\begin{proof}
	First we choose the points $x_1,x_2,x_4,$ and $x_5$.  To form right angles, $x_3$ must lie in the plane containing $x_2$ with normal vector $\overrightarrow{x_1x_2}$, as well as the plane containing $x_4$ with normal vector $\overrightarrow{x_4x_5}$. These two planes must intersect to contain $x_3$. 
	
\indent\textbf{Case 1:} If they are not the same plane, then they intersect on a line. In order for $\angle (x_2, x_3, x_4)$ to be a right angle, $x_3$ must lie on the sphere with antipodal points $x_2$ and $x_4$.  The intersection of the two planes and the sphere is two points, which is the number of choices of $x_3$.  We can choose $x_1,x_2,x_4,$ and $x_5$ in $n^4$ ways, so the total number of choices is $\lesssim n^4$.
	
\indent\textbf{Case 2:} If they are the same plane, observe that $x_1,x_2,x_4,$ and $x_5$ will form a 2-chain, with the $\overline{x_1x_2}$ line parallel to the $\overline{x_4x_5}$ line.  Therefore, $(x_1, x_2, x_4, x_5)$ will form a 2-chain of the type handled by Theorem~\ref{th:parallel}, so there are $\lesssim n^3$ choices of these points. With $n$ choices of $x_3$, this makes a total of $\lesssim n^4$ choices.  
\end{proof}
\begin{theorem}\label{3chains4}
Given a large finite point set $E$ of $n$ points in $\mathbb R^3$,
$$|\Lambda_4(E)|\lesssim n^\frac{13}{3}.$$
\end{theorem}

\begin{proof}
	First we choose the points $x_1,x_2,x_4,x_5,x_6$. To form right angles, $x_3$ must lie in the plane containing $x_2$ with normal vector $\overrightarrow{x_1x_2}$, as well as the plane containing $x_4$ with normal vector $\overrightarrow{x_4x_5}$. These two planes must intersect to contain $x_3$. 
	
\indent\textbf{Case 1:} If they are not the same plane, then they intersect on a line. In order for $\angle (x_2, x_3, x_4)$ to be a right angle, $x_3$ must lie on the sphere with antipodal points $x_2$ and $x_4$.  The intersection of the two planes and the sphere is two points, which is the number of choices of $x_3$.  We can choose $x_4,x_5,$ and $x_6$ in $n^{7/3}$ ways by Theorem~\ref{AS3}, and $x_1$ and $x_2$ in $n^2$ ways. So the total number of choices is $\lesssim n^\frac{13}{3}$.
	
\indent\textbf{Case 2:} If they are the same plane, then we can first choose $x_4,x_5,$ and $x_6$ in $n^{7/3}$ ways by Theorem~\ref{AS3}.  Since the two planes to be equal, $\overrightarrow{x_1x_2}$ is parallel to $\overrightarrow{x_4x_5}$.  So if we choose $x_1$ in $n$ ways, this fixes the line where $x_2$ could lie.  Moreover, $x_2$ must lie in the common plane, so together with the line from $x_1$ there is at most one choice for $x_2$.  We can choose $x_3$ in $n$ ways, so the total number of 4-chains is $\lesssim n^\frac{13}{3}$. 
\end{proof}

In Theorem~\ref{railroad} we have that the number of 4-chains in $\mathbb{R}^2$ is at least $\gtrsim n^4$.  This construction can also be embedded in $\mathbb{R}^3$, so the multiplicative gap between the lower and upper bounds is $n^{1/3}$. For longer chains we have the following recurrence:

\begin{theorem}\label{3chainsTech}
Given a large finite point set $E$ of $n$ points in $\mathbb R^3$, for $k \geq 5,$
$$|\Lambda_k(E)|\lesssim n\cdot |\Lambda_{k-2}(E)|+n^{7/3}|\Lambda_{k-4}(E)|.$$
\end{theorem}

\begin{proof}
	We will split it into two types of $k$-chains: either $\overrightarrow{x_2x_3}$ is parallel to $\overrightarrow{x_5x_6}$ or it is not.
	
\indent\textbf{Case 1:} If they are not parallel, then we can choose $x_1,x_2,$ and $x_3$ in $n^{7/3}$ ways, and $x_5,...,x_k,x_{k+1},x_{k+2}$ in $|\Lambda_{k-4}(E)|$ ways. Now, $x_4$ must lie in the plane with normal vector $\overrightarrow{x_2x_3}$ containing $x_3$, as well as the plane with normal vector $\overrightarrow{x_5x_6}$ containing $x_5$.  Since the normal vectors are not parallel, the planes intersect in a line.  Finally, $x_4$ must lie on the sphere with $x_3$ and $x_5$ as antipodal points.  The sphere intersects the line in at most two places, so there are at most two choices for $x_4$. Hence the bound for $|\Lambda_k(E)|$ is $2\cdot n^{7/3}\cdot |\Lambda_{k-4}(E)|$.  

\indent\textbf{Case 2:} If they are parallel, then we can first choose $x_3,...,x_{k},x_{k+1},x_{k+2}$ in $|\Lambda_{k-2}(E)|$ ways and $x_1$ in $n$ ways. The point $x_2$ must lie on the line through $x_3$ and parallel to $\overrightarrow{x_5x_6}$. It must also lie on the sphere with $x_1$ and $x_3$ as antipodal points.  The line intersects the sphere in at most two places, so there are at most two choices for $x_2$. Hence the bound for $|\Lambda_k(E)|$ is $2\cdot n\cdot |\Lambda_{k-2}(E)|$.
\end{proof}

Here we present the upper bounds for some values of $k$.
\begin{center}
	\begin{tabular}{|c||c|c|c|c|c|c|}
		
		\hline $k$ &5  &6  &7  &8&9  \\\hline
\rule{0pt}{3ex} upper bound of: $n\cdot |\Lambda_{k-2}(E)|$  &$n^\frac{15}{3}$  &$n^\frac{16}{3}$  &$n^\frac{18}{3}$&$n^\frac{20}{3}$&$n^\frac{22}{3}$  \\\hline
\rule{0pt}{3ex} upper bound of: $n^{7/3}\cdot |\Lambda_{k-4}(E)|$&$n^\frac{14}{3}$  &$n^\frac{17}{3}$&$n^\frac{19}{3}$  &$n^\frac{20}{3}$&$n^\frac{22}{3}$ \\\hline   
		
	\end{tabular}
\end{center}
%
%

We now prove Theorem \ref{threeChains} as a corollary of several of the results mentioned above.

\begin{proof}
Again, we let $|\Lambda_k(E)|$ be the maximum number of right angle $k$-chains for a set $E$ of $n$ points in $\mathbb{R}^3$. By Theorem \ref{AS3}, we have that $|\Lambda_1(E)|\lesssim n^\frac{7}{3}.$ For the case $k=2$, we again apply Theorem \ref{AS3} to get a bound on the number of triples of points that form a right angle. Then by choosing the fourth point freely, we get that $|\Lambda_2(E)|\lesssim n|\Lambda_1(E)| \lesssim n^\frac{10}{3}.$ Theorem \ref{3chains3} and Theorem \ref{3chains4} give us $|\Lambda_3(E)|\lesssim n^4$ and $|\Lambda_4(E)|\lesssim n^\frac{13}{3},$ respectively. We then appeal to Theorem \ref{3chainsTech} to get that $|\Lambda_5(E)|\lesssim n^5.$ For $6\leq k\leq 9,$ we can use Theorem \ref{3chainsTech} to verify that we have
\begin{equation}\label{IH}
n|\Lambda_{k-2}(E)|\lesssim n^\frac{7}{3}|\Lambda_{k-4}(E)|.
\end{equation}
Now we show that \eqref{IH} holds for larger values of $k$ by induction. To see this, suppose that \eqref{IH} is true for all $6\leq k\leq m,$ for some $m\geq 9.$ Now consider $|\Lambda_{m+1}(E)|$. By Theorem \ref{3chainsTech}, we have that
\begin{align*}
|\Lambda_{m+1}(E)|&\lesssim n|\Lambda_{(m+1)-2}(E)|+n^\frac{7}{3}|\Lambda_{(m+1)-4}(E)|\\
&= n|\Lambda_{m-1}(E)|+n^\frac{7}{3}|\Lambda_{m-3}(E)|.
\end{align*}
Now, because \eqref{IH} holds for all values of $k$ between 6 and $m$, it applies to $|\Lambda_{m-1}(E)|,$ so we get that 
$$\left(n|\Lambda_{m-1}(E)|\right)+n^\frac{7}{3}|\Lambda_{m-3}(E)|\lesssim \left(n^\frac{7}{3}|\Lambda_{m-3}(E)|\right)+n^\frac{7}{3}|\Lambda_{m-3}(E)|,$$ and by induction we have shown that \eqref{IH} holds for $k\geq 6.$

To finish, we just notice that this recurrence implies that, starting from the value of $\frac{19}{3}$ when $k=7,$ every four consecutive exponents will increase with $k$ by $\frac{1}{3}, \frac{2}{3}, \frac{2}{3},$ and $\frac{2}{3},$ before repeating. So we need the exponents to increase by $\frac{7}{3}$ in four discrete steps. One way to express this is setting, for any $k\geq 6,$
$$|\Lambda_k(E)|\lesssim n^{\frac{1}{3}\left(19+\left\lfloor\frac{7(k-7)}{4}\right\rfloor \right)}.$$
\end{proof}

\section{Angle chains in higher dimensions}\label{4d+chains}


Many discrete geometry problems become trivial in general when we consider them in higher dimensions. One motivation for much of the work in this section is the classical Lenz example for distances. See \cite{BMP} for more on the subject.
\begin{classictheorem}\label{LenzExample}
For $d\geq 4,$ there exists a set of $n$ points in $\mathbb R^d$ with $\approx n^2$ pairs of points that define the same distance.
\end{classictheorem}
\begin{proof}
We construct the set in four dimensions, and it can easily be embedded in higher dimensional spaces. Define $$E:=\left\{(\cos a_1, \sin a_1, 0, 0): a_1=1, \dots, \left\lceil\frac{n}{2}\right\rceil \right\},$$
$$F:=\left\{(0, 0, \cos a_2, \sin a_2): a_2=1, \dots, \left\lfloor\frac{n}{2}\right\rfloor \right\}.$$
Notice that any point in $E$ is at a distance $\sqrt 2$ to any point in $F$. So the union of $E$ and $F$ is a set of $n$ points, with $\gtrsim n^2$ pairs of points that each determine the same distance.
\end{proof}
One key feature to take away from this construction is that sharpness examples for these kinds of questions often rely on low dimensional intersections of varieties in higher dimensions. Getting control on examples like this within general point sets can guide one to better upper bounds, but we believe that they are interesting in their own right. So for the following constructions, we are again looking for low dimensional intersections of high dimensional varieties. We begin with a result from \cite{BMP}, which is also described in \cite{ApfelbaumSharir}.
\begin{classictheorem}
	There is a set of $n$ points in $\mathbb{R}^4$ forming $\approx n^3$ right angles.
\end{classictheorem}
\begin{proof}
For the upper bound, there are at most $n^3$ choices of triples of points.
The lower bound comes from considering
\begin{align*}
x_1&=(-1,0,a_1,0)\\
x_2&=(\cos(a_2),\sin(a_2),0,0)\\
x_3&=(1,0,0,a_3)\\
\end{align*}
where $a_1,a_2,a_3\in\{1,...,\left\lfloor n/3\right\rfloor \}$.  The triple $(x_1, x_2, x_3)$ is a right angle for any choice of $(a_1,a_2,a_3)$, so there are $\gtrsim n^3$ of these chains.   
\end{proof}
Expanding on this construction shows that without further hypotheses, the right angle 2-chain question is trivial in dimension five.
\begin{theorem}
	There is a set of $n$ points in $\mathbb{R}^5$ forming $\approx n^4$ right angle 2-chains.
\end{theorem}
\begin{proof}
For the upper bound, there are trivially at most $n^4$ choices of quadruples of $n$ points.
The lower bound is attained by setting
\begin{align*}
x_1&=(-1,0,a_1,0,1)\\
x_2&=(\cos(a_2),\sin(a_2),0,0,1)\\
x_2&=(1,0,0,\cos(a_3),\sin(a_3)\\
x_4&=(1,0,a_4,0,-1),
\end{align*}
where $a_1,a_2,a_3,a_4\in\{1,...,\left\lfloor n/4\right\rfloor \}$.  The quadruple $(x_1, x_2, x_3, x_4)$ is a right angle 2-chain for any choice of $(a_1,a_2,a_3,a_4)$, so there are $\gtrsim n^4$ of these chains.   
\end{proof}
By following this general idea, we can show that such questions about longer chains are trivial in higher dimensions.
\begin{theorem}\label{6dSharpRight}
For any positive integer $k$, there is a set of $n$ points in $\mathbb{R}^6$ forming $\approx n^{k+2}$ right angle $k$-chains.
\end{theorem}

\begin{proof}
	For the upper bound, there are at most $n^{k+2}$ choices of points.
	Consider the following set of points, in this order:
	\begin{align*}
	x_1&=(-1,0,1,0,\cos(a_1),\sin(a_1))\\
	x_2&=(\cos(a_2),\sin(a_2),1,0,1,0)\\
	x_3&=(1,0,\cos(a_3),\sin(a_3),1,0)\\
	x_4&=(1,0,-1,0,\cos(a_4),\sin(a_4))\\
	x_5&=(\cos(a_5),\sin(a_5),-1,0,-1,0)\\
	x_6&=(-1,0,\cos(a_6),\sin(a_6),-1,0)
	\end{align*}
	where the $a_i's$ are positive integers between $1$ and $n/6$.  

For every ordered triple $(x_i,x_{i+1},x_{i+2})$, where $i-1\in \mathbb{Z}/6\mathbb{Z}$, these points form a right angle.  Therefore, for any positive integer $k$, there are $n/6$ choices of the form $x_1$, and $n/6$ of $x_2$, and then of $x_3,x_4,x_5,x_6$, and then another $n/6$ going back to $x_1,$ and so on.  This forms a total of $\gtrsim n^{k+2}$ instances of $k$-chains.

\end{proof}

\begin{theorem}\label{6dSharpAcute}
	For a range of choices of $k$ non-right-angles $(\alpha_1,...,\alpha_k)$ there is a set of $n$ points in $\mathbb{R}^6$ forming $\approx n^{k+2}$ instances of a $k$-chain with angles $(\alpha_1,...,\alpha_k)$.  Each angle $\alpha_i$ can be chosen in some interval  $(\beta_i,\pi/2)$, where $\beta_i$ is some number in $(0,\pi/2)$ and depends on the choice of $\alpha_{i-2}$ and $\alpha_{i-1}$.
\end{theorem}

\begin{proof}
	Consider the following $k+2$ points
	
	\begin{align*}
	&x_1=(c_1\cos(a_1),c_1\sin(a_1),0,0,0,0)\\
	&x_2=(0,0,c_2\cos(a_2),c_2\sin(a_2),0,0)\\
	&x_3=(0,0,0,0,c_3\cos(a_3),c_3\sin(a_3))\\
	&\vdots\\
	&x_{k+2}.
	\end{align*}
	
Here $x_i$ has $c_i\cos(a_i)$ in the $2i-1$ (mod 6) entry, $c_i\sin(a_i)$ in the entry that is $2i$ (mod 6), and 0 in the other entries. To simplify, set $c_1=c_3$ and $c_2=1$.   So $\alpha_1=\arccos\left(\dfrac{1}{1+c_3^2}\right).$ We can choose $c_3$ to be any number in $(0,\infty),$ so $\alpha_1$ can be any angle in $(0,\pi/2)$. We set
$$\alpha_2=\arccos\left(\dfrac{c_3^2}{\sqrt{1+c_3^2}\sqrt{c_3^2+c_4^2}}\right)$$
$$=\arccos\left(\dfrac{c_3}{\sqrt{1+c_3^2}\sqrt{1+c_4^2/c_3^2}}\right).$$
Our choice of $c_4$ in $(0,\infty)$ makes $\sqrt{1+c_4^2/c_3^2}$ some number in $(1,\infty)$, so $\alpha_2>\arccos\left(\dfrac{c_3}{\sqrt{1+c_3^2}}\right)$.
	
For a general $i$, we have
$$\alpha_i=\arccos\left(\dfrac{c_{i+1}^2}{\sqrt{c_{i}^2+c_{i+1}^2}\sqrt{c_{i+2}^2+c_{i+1}^2}}\right)$$$$=\arccos\left(\dfrac{c_{i+1}}{\sqrt{c_{i+1}^2+c_i^2}\sqrt{1+c_{i+2}^2/c_{i+1}^2}}\right).$$
Our choice of $c_{i+2}$ in $(0,\infty)$ makes $\sqrt{1+c_{i+2}^2/c_{i+1}^2}$ some number in $(1,\infty)$, so $\alpha_i>\arccos\left(\dfrac{c_{i+1}}{\sqrt{c_i^2+c_{i+1}^2}}\right)$.
	
We define $\beta_i:=\arccos\left(\dfrac{c_{i+1}}{\sqrt{c_{i}^2+c_{i+1}^2}}\right)$. For angles $(\alpha_1,...,\alpha_k)$ where $$\alpha_i=\arccos\left(\dfrac{c_{i+1}^2}{\sqrt{c_{i}^2+c_{i+1}^2}\sqrt{c_{i+2}^2+c_{i+1}^2}}\right),$$ every ordered triple $(x_{i-1},x_i,x_{i+1})$ makes angle $\alpha_i$.   Therefore, for any positive integer $k$, there are $n/(k+2)$ choices of $(k+1)$-tuples of the form $x_1,...,x_{k+2}$. This gives a total $\gtrsim n^{k+2}$ instances of a $k$-chain whose angles are $(\alpha_1,...,\alpha_k)$.
\end{proof}

\begin{example}
As an example of this theorem, we can choose $c_1=c_2=\cdots c_{k+2}=1$, which makes $\alpha_1=\cdots=\alpha_k=\pi/3$.
\end{example}

\section{Pinned results}\label{pinnedSection}
\subsection{Proof of Proposition \ref{middlePinned}}
\begin{proof}
Find two rays that determine the angle $\alpha.$ Call their shared endpoint $x_2,$ and let it be in our point set. Then arrange half of the remaining points along one ray, and the rest along the other. So there are roughly $n/2$ points on the first ray, giving us $\approx n$ choices for $x_1,$ and roughly $n/2$ points on the second ray, giving us $\approx n$ choices for $x_3.$ In total, we have $\approx n^2$ triples of the form $(x_1, x_2, x_3)$ that determine our angle $\alpha$ with the same choice of middle point.
\end{proof}
\subsection{Proof of Theorem \ref{th:2dpin}}
Henceforth, we focus on pinned results where the choice of $x_1$ remains fixed. For pinned angles in $\mathbb{R}^2$ we first recall the statement of Theorem \ref{th:2dpin}.
\vspace{2mm}

\noindent {\bf Theorem \ref{th:2dpin}.}
\emph{For any $n$ points in $\mathbb{R}^2$, and any fixed angle $0<\alpha<\pi$, there are at most $\lesssim n^{4/3}$ triples of points with angle $\alpha$ starting from the origin, and this is sharp.}

\begin{proof}
Fix some point set $P$ of $n$ points in $\mathbb{R}^2$.  We'll count an $\alpha$ angle starting at the origin and including two of the $n$ points as follows:  
Pick a point $x\in P$.  There are two lines (one line when $\alpha=\pi/2$) through $x$, such that the origin, then $x$, and then any point on one of these lines forms an angle $\alpha$, and is the only choice of lines forming angle $\alpha$.  For each of these $n$ points, call $L$ the collection of these lines. 
	
In other words, an $\alpha$ angle occurs if and only if a point of $P$ lies on a line of $L$, except for the point $x$ was the middle point of the angle.  So the number of angles is the incidence of the points $P$ and lines $L$, minus $P$.  Since $|P|=n$ and $|L|=2n$ ($|L|=n$ when $\alpha=\pi/2$), we can appeal to Theorem \ref{szemTrot} to conclude that the number of these incidences is $\lesssim n^{4/3}.$    

We now turn our attention to the sharpness, which is realized with the following construction, motivated by classical sharpness examples for the Theorem \ref{szemTrot}:

	\begin{align*}
	P&=\left\{(a,b)\in \mathbb{Z}^2:1\le a\le n^{1/3}, 1\le b\le n^{2/3}\right\}\\
	L&=\left\{y=ax+b:(a,b)\in \mathbb{Z}^2,1\le a\le n^{1/3}, 1\le b\le \frac{n^{2/3}}{2}\right\}
	\end{align*}
	
	Note that $|P|=n$, $|L|=n/2$, and each line of $L$ is incident to $n^{1/3}$ points.
	For each line in $L$, there are two pivot points (one point if $\alpha=\pi/2$) on the line (not necessarily coming from $P$) forming an angle of $\alpha$ with origin.  Combined through these two pivot points on $L$, there are $n^{1/3}$ points on the line intersecting $P$ forming an $\alpha$ angle. For each line, call the collection of these points $Q$. Since $|Q|=n$ (or $n/2$ if $\alpha=\pi/2$). So $|P\cup Q|\approx n$.  So we have a collection of $\approx n$ points forming $|L|n^{1/3}\approx n^{4/3}$ instances of $\alpha$-angles.  
	
\end{proof}

For pinned chains in $\mathbb{R}^2$ we first recall the upper bound of Theorem \ref{th:2dchainpinupper}.

\noindent {\bf Theorem \ref{th:2dchainpinupper}.}
\emph{For any set $E$ of $n$ points in $\mathbb{R}^2$, integer $k\ge 2$ and angles $(\alpha_1,...,\alpha_k)$, the number of $k$-chains of type $(\alpha_1,...,\alpha_k)$ starting from the origin is}

	$$\lesssim \begin{cases}
	n^{\frac{k-1}{2}+1}\log(n)& k\text{ is odd}\\
	n^{\frac{k}{2}+1}& k\text{ is even}
	\end{cases}$$

\begin{proof}
	For $k=2,$ we can choose points $x_3$ and $x_4$ in $\lesssim n^2$ ways.  There are two lines through $x_3$ that form an angle of $\alpha_2$ with the $\overline{x_3x_4}$. So $x_2$ must lie on one of these lines.  In order for $\angle x_1x_2x_3=\alpha_1$, there is only one choice of $x_2$ on each of the two lines.  Since $x_1=(0,0)$ is fixed, there are a total of $\lesssim n^2$ choices.
	
	For $k\geq 3,$ we can choose an $k$-chain (unpinned) of type $(\alpha_3,...,\alpha_k)$ in $\lesssim n^{\frac{k-1}{2}+1}\log(n)$ ways if $k$ is odd and $\lesssim n^{\frac{k}{2}+1}$ ways if $k$ is even. By the same argument above, there are two choices of $x_2$, and $x_1$ is fixed.  This concludes the proof.
\end{proof}

We next recall the lower bound from Theorem \ref{th:2dchainpinlower}.
\vspace{2mm}

\noindent {\bf Theorem \ref{th:2dchainpinlower}.}
\emph{For any ordered set of angles $(\alpha_1,...,\alpha_k)$, there is a set of $n$ points in $\mathbb{R}^2$ forming $\gtrsim n^{\left\lfloor{\frac k2}\right\rfloor+1}$ instances of $k$-angle chains of type $(\alpha_1,...,\alpha_k)$ starting at the origin.} 

\begin{proof}
	Given $(\alpha_1, \dots, \alpha_k),$ set $m=\lfloor 2n/k \rfloor.$ Arrange $m$ points, $p_1, \dots, p_m,$ in order away from the origin, along the $x$-axis, which we will call $\ell_1$, then draw the line $\ell_2$ so that it intersects the line $\ell_1$ at an angle of $\alpha_1-\alpha_2$ at the origin. Now for each of the points $p_j$ on $\ell_1,$ except for $p_m,$ the one furthest from the intersection of $\ell_1$ and $\ell_2$, put a point, $p_{m+j}$, on $\ell_2$ so that $\angle(p_{j'},p_j, p_{m+j})=\alpha_1$ for any $j'>j.$ Notice that by construction, we will also have that $\angle(p_j, p_{m+j}, p_{m+j''})=\alpha_2$ for any $j''<j.$ Now rotate the whole set so that $\ell_2$ coincides with the $x$-axis, and draw $\ell_3$ so that it intersects $\ell_2$ at the origin in an angle of $\alpha_3-\alpha_4$, and continue. 
	Now start at one fixed point. At each stage, choose points so that they don't overlap with other points just in case the set wraps around the origin after a number of rotations. Notice that we can form the desired number of angle $k$-chains by picking $x_1$ and $x_2$ from $\approx n^2$ pairs of points from $\ell_1$, then a fixed point $x_3$, determined by $x_2$, from $\ell_2$, then choosing $x_4$ from $\approx n$ points on $\ell_2$ so that $\angle(x_2,x_3,x_4)=\alpha_2,$ and so on.
\end{proof}

%
%

In $\mathbb{R}^3$ we show that already for a single right angle the problem is trivial without further restrictions.

\begin{theorem}\label{3pinned}
There is a set of $\approx n$ points in $\mathbb{R}^3$ forming $\approx n^2$ right angles starting at the origin.  
\end{theorem}
\begin{proof}
Let $P$ be the set of $n/2$ points $y_i \in \mathbb R^3$ with coordinates $(1+\cos i, \sin i, 0),$ where $i=1, \dots, n/2.$ Let $Q$ be the set of $n/2$ points $z_j \in \mathbb R^3$ with coordinates $(2, 0 ,j),$ where $j=1, \dots, n/2.$ Now choose an arbitrary pair of $y_i \in P$ and $z_j \in Q.$ We can verify that  the origin, $y_i$, and $z_j$ form a right angle by computing the dot product of the vector $v=y_i-(0,0,0)$ with the vector $w=y_i-z_j.$
\begin{align*}
v\cdot w &= (1+\cos i, \sin i, 0) \cdot (-1+\cos i, \sin i, j)\\
&= (-1 + \cos i - \cos i + \cos^2 i) + (\sin^2 i) + 0\\
& = -1 + (\sin^2 i + \cos^2 i) = 0.
\end{align*}
The idea is that the planes normal to each choice of $y_i$ containing that particular point $y_i$ all meet at the line containing $Q.$ One could in principle choose any set of about $n$ points on this circle containing $P$ and any set of about $n$ points on the line containing $Q$ to get the same result. 
\end{proof}

We also have the following pinned versions of Theorems \ref{6dSharpRight} and \ref{6dSharpAcute}.

\begin{theorem}
	For any positive integer $k$, there is a set of $n$ points in $\mathbb{R}^6$ forming $\approx n^{k+1}$ right angle $k$-chains starting from the origin.
\end{theorem}

\begin{proof}
	For a set of $n$ points, there are $k+1$ choices of $n$ points not including the origin to form a $k$ chain.  So there are  $\lesssim n^{k+1}$ $k$ angle chains starting from the origin.
	
	For the lower bound, consider the following set of points, in this order:
	\begin{align*}
	x_1&=(0,0,2,0,1+\cos(a_1),\sin(a_1))\\
	x_2&=(1+\cos(a_2),\sin(a_2),2,0,2,0)\\
	x_3&=(2,0,1+\cos(a_3),\sin(a_3),2,0)\\
	x_4&=(2,0,0,0,1+\cos(a_4),\sin(a_4))\\
	x_5&=(1+\cos(a_5),\sin(a_5),0,0,0,0)\\
	x_6&=(0,0,1+\cos(a_6),\sin(a_6),0,0)
	\end{align*}
	where the $a_i's$ are positive integers between $1$ and $n/6$.  
	
	For every ordered triple $(x_i,x_{i+1},x_{i+2})$, where $i-1\in \mathbb{Z}/6\mathbb{Z}$, these points form a right angle.  Therefore, for any positive integer $k$, there are $n/6$ choices of the form $x_1$, and $n/6$ of $x_2$, and then of $x_3,x_4,x_5,x_6$, and then another $n/6$ going back to $x_1,$ etc.  This forms a total of $\gtrsim n^{k+1}$  $k$-chains starting from the origin.
	
	\end{proof}
\begin{theorem}
	For a range of choices of $k$ non-right-angles $(\alpha_1,...,\alpha_k)$ there is a set of $n$ points in $\mathbb{R}^6$ forming $\approx n^{k+1}$ instances of $k$-chains with angles $(\alpha_1,...,\alpha_k)$ starting from the origin.  Each angle $\alpha_i$ can be chosen in some interval  $(\beta_i,\pi/2)$, where $\beta_i$ is some number in $(0,\pi/2)$ and depends on the choice of $\alpha_{i-2}$ and $\alpha_{i-1}$.
\end{theorem}

\begin{proof}
	Consider the same $k+2$ points as in Theorem \ref{6dSharpAcute}:
	
	\begin{align*}
	&x_1=(c_1\cos(a_1),c_1\sin(a_1),0,0,0,0)\\
	&x_2=(0,0,c_2\cos(a_2),c_2\sin(a_2),0,0)\\
	&x_3=(0,0,0,0,c_3\cos(a_3),c_3\sin(a_3))\\
	&\vdots\\
	&x_{k+2}.
	\end{align*}
Here $x_i$ has $c_i\cos(a_i)$ in the $2i-1$ (mod 6) entry, $c_i\sin(a_i)$ in the entry that is $2i$ (mod 6), and 0 in the other entries. Setting $c_1=0$ makes $x_1$ the origin.  As in the proof of Theorem \ref{6dSharpAcute}, we can define
$$\beta_i:=\arccos\left(\dfrac{c_{i+1}}{\sqrt{c_{i}^2+c_{i+1}^2}}\right).$$
For angles $(\alpha_1,...,\alpha_k)$ where
$$\alpha_i=\arccos\left(\dfrac{c_{i+1}^2}{\sqrt{c_{i}^2+c_{i+1}^2}\sqrt{c_{i+2}^2+c_{i+1}^2}}\right),$$
every ordered triple $(x_i,x_{i+1},x_{i+2})$ makes angle $\alpha_i$. Therefore, for any positive integer $k$, there are $n/(k+1)$ choices of $(k+1)$-tuples of the form $x_2,..., x_{k+2}$. This gives a total $\gtrsim n^{k+1}$ instances of $k$-chains whose angles are $(\alpha_1,...,\alpha_k)$. 
\end{proof}

\end{document}